\documentclass[12pt, a4paper]{amsart}

\usepackage{amsmath,amssymb}

\usepackage[colorlinks=true, pdfstartview=FitV, linkcolor=blue, citecolor=blue, urlcolor=blue]{hyperref}
\usepackage{color}
\usepackage{enumitem}

\newtheorem{theorem}{Theorem}[section]
\newtheorem{lemma}[theorem]{Lemma}
\newtheorem{proposition}[theorem]{Proposition}

\theoremstyle{definition}

\theoremstyle{remark}
\newtheorem{remark}[theorem]{Remark}
\numberwithin{equation}{section}

\newcommand{\N}{\mathbb{N}}

\newcommand{\R}{\mathbb{R}}

\newcommand{\rn}{\mathbb{R}^n}
\newcommand{\sn}{\mathbb{S}^{n-1}}

\begin{document}

\title[ ]
{Boundedness properties in a family of weighted Morrey spaces with emphasis on power weights}

\author{Javier Duoandikoetxea and Marcel Rosenthal} 
\address{Universidad del Pa\'is Vasco/Euskal Herriko Unibertsitatea UPV/EHU, Departamento de Matem\'a\-ti\-cas, 
Apartado 644, 48080 Bilbao, Spain}

\email{javier.duoandikoetxea@ehu.eus, marcel.rosenthal@uni-jena.de} 


\begin{abstract}
We define a scale of weighted Morrey spaces which contains different weighted versions appearing in the literature. This allows us to obtain weighted estimates for operators in a unified way. In general, we obtain results for weights of the form $|x|^\alpha w(x)$ with $w\in A_p$ and nonnegative $\alpha$. We study particularly some properties of power-weighted spaces and in the case of the Hardy-Littlewood maximal operator our results for such spaces are sharp. By using extrapolation techniques the results are given in abstract form in such a way that they are automatically obtained for many operators. 
\end{abstract}

\maketitle

\section{Introduction}\label{intro}

We consider a scale of weighted Morrey spaces. The weight $w$ is a nonnegative measurable function.
For $0< p<\infty$, $\lambda_1\in \R$, $\lambda_2 \in \R$, let  $\mathcal L^{p,\lambda}(w)$, $\lambda=(\lambda_1,\lambda_2)$, be the Morrey space formed by the collection of all measurable functions $f$ such that 
\begin{equation}\label{normdef}
\aligned
&\|f\|_{\mathcal L^{p,\lambda}(w)} := \sup_{x\in \rn, r>0}\left(\frac 1{r^{\lambda_1}\ w(B(x,r))^{\lambda_2/n}}\int_{B(x,r)}|f|^p w\right)^{1/p}\\
																 &\cong 
																\sup_{x\in \rn, r>0}\left(\frac 1{|B(x,r)|^{\lambda_1/n}\ w(B(x,r))^{\lambda_2/n}}\int_{B(x,r)}|f|^p w\right)^{1/p} <\infty.
\endaligned
\end{equation}
(Here and in what follows $w$ is locally integrable and $w(A)$ stands for the integral of $w$ over $A$ and $|A|$ for the Lebesgue measure of $A$.) 
We also consider the weak Morrey space $W\mathcal L^{p,\lambda}(w)$, for which 
\begin{equation}\label{normdefw}
\|f\|_{W\mathcal L^{p,\lambda}(w)}:=\sup_{x\in \rn, r>0, t>0}\left(\frac {t^p\,w(\{y\in B(x,r): |f(y)|>t\})}{|B(x,r)|^{\lambda_1/n}\ w(B(x,r))^{\lambda_2/n}}\right)^{1/p}<\infty.
\end{equation}
Clearly, $\mathcal L^{p,\lambda}(w)\subset W\mathcal L^{p,\lambda}(w)$. Several conditions depending on $w$ will be imposed to $\lambda_1$ and $\lambda_2$ so that the involved Morrey spaces are not reduced to the zero function. 

This general formulation encompasses three interesting special cases of weighted Morrey spaces that have been considered in the literature, namely,
\begin{itemize}
\item the case $\lambda_1=0$ and $0<\lambda_2<n$ corresponds to the spaces introduced by Komori and Shirai in \cite{KS09};
\item  the case $0<\lambda_1<n$ and $\lambda_2=0$ corresponds to the spaces introduced by N.~Samko in \cite{Sam09};
\item the case $\lambda_1<0$ and $\lambda_2=n$ corresponds to the spaces considered by Poelhuis and Torchinsky in \cite{PT15}, built as the weighted versions of the spaces in \cite{SST11} (although a more general function of $r$ appears there).
\end{itemize}

There are many papers devoted to the study of the boundedness of operators in weighted Morrey spaces, mainly defined in the way of the first two cases above, with some generalizations. In \cite{DR18} we studied the extension of inequalities on weighted Lebesgue spaces with $A_p$ weights to Morrey spaces of the first two types mentioned in the previous paragraph (see also \cite{RS16} for the first type), using techniques which are related to the extrapolation results in the Lebesgue setting. This provides immediately the boundedness on weighted Morrey spaces of a variety of operators. The results in \cite{DR18} were extended in \cite{DR19} to weights of the form $|x|^\alpha w(x)$, but only for the case of Samko-type spaces ($\lambda_2=0$). For fixed $p$ such weights are beyond $A_p$. In this paper we consider the same type of weights for the Morrey spaces defined above and results beyond $A_p$ are thus obtained for all of them. In particular, we emphasize the case of power weights for which the results are optimal.     

In Section \ref{bi} we discuss several properties of the spaces in the family and, in particular, we give conditions on $\lambda_1$ and $\lambda_2$ (depending on the weight) which guarantee the existence of nontrivial functions in the weighted Morrey space. For power weights we fully describe such conditions (Proposition \ref{betaspaces}). An interesting observation (see Proposition \ref{reduct}) is the possibility of restricting the definition of the norm in \eqref{normdef} to special balls, namely, to $B(x,r)$ with $r\le |x|/4$. This is particularly helpful for power weights and can be used to identify power weighted spaces with different values of $\lambda_1$ and $\lambda_2$ (for instance,  in \eqref{samkomori} below Komori-Shirai and N. Samko type spaces are identified). 

In Section \ref{hiru} we deal with the Hardy-Littlewood maximal operator. When $\lambda_1$ and $\lambda_2$ are nonnegative we prove that the boundedness of the  on $L^{p,\lambda}(w)$ implies that $w$ must belong to a certain Muckenhoupt class $A_q$, with $q=q(n,p,\lambda)>p$. In the case of power weights with positive exponent this class is sharp. The sharp range for power weights is also obtained in the cases in which $\lambda_1$ is negative, thus covering the interesting case $\lambda_2=n$. 

In Section \ref{bost} we work in the setting of the extrapolation theorem: if an operator is bounded on $L^{p_0}(w)$ for all $w\in A_{p_0}$, we deduce that it is bounded on $\mathcal L^{p,\lambda}(|x|^\alpha w)$ for $w\in A_p$ and $\alpha$ in a range depending on $\lambda$ and $n$. This result immediately applies to many operators as can be seen in \cite[Section 5]{DR18}. We also show a variant of the extrapolation theorem with a weaker assumption that is useful for many other operators. 

We are using power weights of the form $|x|^\beta$, centered at the origin. But it is clear that the same results hold for powers of the form $|x-x_0|^\beta$, with $x_0\in \rn$. 


\section{Preliminary results and significant properties of (power) weighted Morrey spaces}\label{bi}

Let 
$w\in L_1^\textrm{loc}(\rn)$ with $w>0$ almost everywhere. We say that $w$ is a \textit{Muckenhoupt weight} in the class $A_p$ for $1<p<\infty$ if
\begin{equation} \label{defap}
  [w]_{A_p}\equiv 
	\sup_B \frac{w(B)}{|B|} \left( \frac{w^{1-p'}(B)}{|B|} \right)^{p-1}<\infty,
\end{equation}
where the supremum is taken over all Euclidean balls $B$ in $\rn$. The $A_p$ constant of $w$ is the quantity  $[w]_{A_p}$ of the definition.
We say that $w$  is in $A_1$ if, for any Euclidean ball $B$,
\begin{equation} \label{A1}
  \frac{w(B)}{|B|}\le c w(x) \text{ for almost all } x\in B.
\end{equation}
The $A_1$ constant of $w$, denoted by  $[w]_{A_1}$, is the smallest constant $c$ for which the inequality holds. 

We say that a nonnegative locally integrable function $w$ on $\rn$ belongs to the \textit{reverse H\"older class} $RH_\sigma$ for $1<\sigma<\infty$ if it satisfies the \textit{reverse H\"older inequality} with exponent $\sigma$, that is,  \begin{equation*}
  \left(\frac{1}{|B|}\int_{B} w(x)^\sigma dx \right)^\frac{1}{\sigma}\le \frac{C}{|B|} \int_{B} w(x) dx,
\end{equation*}
where the constant $C$ is independent of the Euclidean ball $B\subset \rn$. 
The $A_p$ classes are increasing with $p$ and the union of all of them is denoted by $A_\infty$. The $RH_\sigma$ classes are decreasing with $\sigma$. Their union for $1<\sigma<\infty$ coincides with $A_\infty$. That is, every $A_p$ weight is in some class $RH_\sigma$ and every weight satisfying a reverse H\"older inequality is in some $A_p$.

To each $w\in A_\infty$ one can associate $\sigma_w\in (1,\infty]$ in the following way:
\begin{equation}\label{bestrhe}
\sigma_w=\sup\{\sigma: w\in RH_\sigma\}.
\end{equation}
According to Gehring's lemma, if $w\in RH_\sigma$, then $w\in RH_{\sigma+\epsilon}$ for some $\epsilon>0$, hence the supremum in \eqref{bestrhe} is not attained.

It was proved in \cite{SW85} that  $w\in RH_\sigma$ if and only if  $w^\sigma\in A_\infty$. By factorization the product of an $A_\infty$ weight and a negative power of an $A_1$ weight is in $A_\infty$. Positive powers of $|x|$ are negative powers of $A_1$ weights because $|x|^\gamma\in A_1$ for $-n<\gamma\le 0$. Therefore, if $w\in RH_\sigma$ and $\alpha>0$, then $|x|^\alpha w\in RH_\sigma$. Using the notation of \eqref{bestrhe}, $\sigma_{|x|^\alpha w}\ge \sigma_w$.

Two results for weights that we use repeatedly are the following.
\begin{itemize}
\item[(i)] If $Mh<\infty$ a.e., then $(Mh)^{1/s}\in A_1$ and its $A_1$ constant depends on $s$, but not on $h$. Moreover, $(Mh)^{1/s}\in A_1\cap RH_\sigma$ if $s>\sigma$.
\item[(ii)] Let $\sigma<\sigma_w$ (that is, $w\in RH_\sigma$). For any ball $B$ and any measurable $E\subset B$ it holds that 
\begin{equation}\label{ineqRH}
 \frac{w(E)}{w(B)}\le c \left(\frac{|E|}{|B|}\right)^{1/\sigma'}. 
\end{equation}
\end{itemize}

Under certain conditions on the weight, the balls considered in the definition of the norm can be restricted to a specific type. We study this reduction in a more general formulation. For $1\le p<\infty$ and nonnegative $u$ and $w$, define 
\begin{equation} \label{genuv}
\|f\|_{\mathcal L^{p}(u,w)} := \sup_{B}\left(\frac 1{u(B)}\int_{B}|f|^p w\right)^{1/p},
\end{equation}
where the supremum is taken over all the balls in $\rn$. Let us classify the balls into three types: 
\begin{itemize}
\item \textit{Type I}: balls centered at the origin;
\item \textit{Type II}: balls $B(x,r)$ centered at $x\ne 0$ and such that $r\le |x|/4$;
\item \textit{Type III}: balls $B(x,r)$ centered at $x\ne 0$ and such that $r> |x|/4$.
\end{itemize}

\begin{proposition}\label{reduct}
\begin{enumerate}
\item If $u$ is doubling, we get a value equivalent to $\|f\|_{\mathcal L^{p}(u,w)}$ by taking the supremum only on balls of type I and II.
\item If moreover $u$ satisfies $\sum_{j=0}^\infty u(B(0,a^jR))\le C u(B(0,R))$ for $a<1$ and $C$ independent of $R$, then we can further restrict the supremum to balls of type II.
\end{enumerate}
\end{proposition}

The second reduction is inspired by Lemma 1.1 of \cite{NST19}.

\begin{proof}
(1) For balls of type III we have $B(x,r)\subset B(0, |x|+r)\subset B(x, 9r)$. Set $R=|x|+r$. Then
\begin{equation*}
\aligned
\frac 1{u(B(x,r))}\int_{B(x,r)}|f|^p w&\le \frac {u(B(0, R))}{u(B(x,r))}\frac 1{u(B(0, R)))}\int_{u(B(0, R))}|f|^p w\\
&\le \frac {u(B(x, 9r))}{u(B(x,r))}\frac 1{u(B(0, R)))}\int_{u(B(0, R))}|f|^p w,
\endaligned
\end{equation*}
 and the first quotient in the last term is bounded if $u$ is doubling.
 
 (2) Fix $a\in (0,1)$. Then 
 \[
B(0, R)=\bigcup_{j=0}^\infty B(0, a^jR)\setminus B(0, a^{j+1}R).
 \]
Set $A_j:= B(0, a^jR)\setminus B(0, a^{j+1}R)$. There exists $C(n)$ depending only on the dimension $n$ such that the annulus $A_j$ can be covered by $C(n)$ balls $B_{k,j}:=B(y_{k,j},r_j)$, $k=1,\dots,C(n)$, for which $|y_{k,j}|=a^j(1+a)/2$ (that is, $y_{k,j}$ is equidistant from the spheres of the boundary of the annulus) and $r_j=a^j(1-a)$ (the width of the annulus). Consequently,
 \begin{equation*}
\aligned
&\frac 1{u(B(0, R))}\int_{u(B(0, R))}|f|^p w =\frac 1{u(B(0, R))}\sum_{j=0}^\infty \int_{A_j}|f|^p w\\
&\qquad\le \frac 1{u(B(0, R))}\,C(n)\sum_{j=0}^\infty u(B(0, a^jR)) \sup_{k,j}\frac 1{u(B_{k,j})}\int_{B_{k,j}}|f|^p w.
\endaligned
\end{equation*}
The balls $B_{k,j}$ are of type II if $4 a^j(1-a)\le a^j(1+a)/2$ and this holds for $a=7/9$, for instance.
\end{proof}

\begin{remark}\label{remred}
For the spaces $\mathcal L^{p,\lambda}(w)$ considered in this paper we have $u(B(x,r))= r^{\lambda_1}w(B(x,r))^{\lambda_2/n}$, which is doubling if $w$ is doubling. Moreover, if $w\in RH_\sigma$ and $\lambda_2\ge 0$, using \eqref{ineqRH}  we have
\begin{equation}\label{redsigma}
u(B(0,a^jR))\le C a^{j(\lambda_1+\lambda_2/\sigma')} u(B(0,R))
\end{equation}
and the condition required in part 2 of the proposition holds for $\lambda_1\sigma'+\lambda_2>0$. Using the definition of $\sigma_w$ of \eqref{bestrhe} the reduction to balls of type II is possible as far as $\lambda_1\sigma_w'+\lambda_2> 0$. Notice that this condition is always satisfied when both $\lambda_1$ and $\lambda_2$ are nonnegative. The condition is also sufficient for weights $|x|^\alpha w$ (with $\alpha>0$) because $\sigma_{|x|^\alpha w}\ge \sigma_w$ as mentioned above. 

In the case of power weights, $w(x)=|x|^\beta$, $\beta>-n$, we have the following result: the reduction to balls of type II is possible if  $\lambda_1+\lambda_2(1+\beta/n)>0$. Indeed, $\lambda_1+\lambda_2(1+\beta/n)$ is the exponent in \eqref{redsigma} in this case. Notice that here we are not assuming that $\lambda_2$ is nonnegative. For $\lambda_2\ge 0$ and $-n<\beta\le 0$ this result is the same as before because $\sigma_w'= n/(n+\beta)$; for $\beta>0$ it is better, because $\sigma_w'=1$.
\end{remark}
 
 \begin{remark}
If we consider weak-type spaces by modifying the definition \eqref{genuv}, it is immediate to check that the same proof works. Then the statement of Proposition \ref{reduct} remains true for weak-type Morrey spaces. 
\end{remark}

The restriction $\lambda_1+\lambda_2<n$ is natural as shown in the following proposition. 
\begin{proposition}\label{special}
Let $w$ be locally integrable and $w(x)>0$ a.e. Then if $\lambda_1+\lambda_2>n$, the space $\mathcal L^{p,\lambda}(w)$ only contains the null function. If $\lambda_1+\lambda_2=n$ and $0\le \lambda_2\le n$, then $\mathcal L^{p,\lambda}(w)=\{f: |f|^p w^{1-\lambda_2/n}\in L^\infty\}$.
\end{proposition}

\begin{proof}
We can write
\begin{equation*}
\frac 1{r^{\lambda_1}\ w(B(x,r))^{\lambda_2/n}}\int_{B(x,r)}|f|^pw=\frac {r^{n-\lambda_1-\lambda_2}}{\left(\dfrac {w(B(x,r)}{r^n}\right)^{\lambda_2/n}}\frac 1{r^n}\int_{B(x,r)}|f|^pw.
\end{equation*}
If $\lambda_1+\lambda_2=n$, then the limit when $r$ goes to zero is $|f(x)|^p w(x)^{1-\lambda_2/n}$ a.e. and therefore this function is in $L^\infty$. The converse also holds: if $|f|^p w^{1-\lambda_2/n}\in L^\infty$, then $f$ is in the Morrey space. Indeed,
\begin{equation*}
\frac 1{r^{n-\lambda_2}\ w(B)^{\lambda_2/n}}\int_{B}|f|^pw\le \||f|^p w^{1-\lambda_2/n}\|_\infty\, \dfrac {\int_B w^{\lambda_2/n}}{r^{n-\lambda_2}\ w(B)^{\lambda_2/n}},
\end{equation*}
and the result follows from H\"older's inequality.

If $\lambda_1+\lambda_2>n$, the exponent of $r^{n-\lambda_1-\lambda_2}$ is negative and the limit when $r$ goes to zero is infinity unless $f$ is identically zero.
\end{proof}

Let us prove that for nonnegative values of $\lambda_1$ and $\lambda_2$ there is a Lebesgue space embedded in the Morrey spaces. 
\begin{lemma}\label{emb0.5}
Let $1\le p<\infty$, $0\le \lambda_1, \lambda_2 <n$ and $\lambda_1 + \lambda_2 < n$.  The embedding \[L^{\frac {pn}{n-\lambda_1-\lambda_2}}(w^{\frac{n-\lambda_2}{n-\lambda_1-\lambda_2}})\hookrightarrow \mathcal L^{p,\lambda}(w)\] holds with constant depending only on $n$, $\lambda$ and $p$, not on $w$.
\end{lemma}

\begin{proof}
Let $B$ be a ball. Then, by H\"older's inequality with three functions,
\begin{equation*}
\aligned
&\frac 1{|B|^{\lambda_1/n}\ w(B)^{\lambda_2/n}}\int_{B}|f|^pw
\\ &\le  \frac 1{|B|^{\lambda_1/n}\ w(B)^{\lambda_2/n}}\left(\int_{B}\left(|f|^pw^\theta\right)^{s_1}\right)^\frac{1}{s_1} 
\left(\int_{B}w^{(1-\theta){s_2}}\right)^\frac{1}{s_2} \left(\int_{B}1\right)^\frac{1}{s_3} 
\\ &\le  \|f\|_{L^{\frac {pn}{n-\lambda_1-\lambda_2}}(w^{\frac{n-\lambda_2}{n-\lambda_1-\lambda_2}})}^p,  
\endaligned
\end{equation*}
where we choose $s_3=n/\lambda_1$, $s_2=n/\lambda_2$, $(1-\theta)s_2=1$ (that is, $\theta=1-\lambda_2/n$) and $1/s_1+1/s_2+1/s_3=1$  (that is, $1/s_1=1-(\lambda_1+\lambda_2)/n$).
\end{proof}

When $\lambda_1, \lambda_2\ge 0$ and at least one of them is not zero, the term appearing in front of the integral in Definition \eqref{normdef} is decreasing with $r$ and tends to zero as $r$ tends to infinity. This is not necessarily true if one of the values of $\lambda_1$ or  $\lambda_2$ is negative, but it holds if $\lambda_2$ is nonnegative and we impose the condition $\lambda_1 {\sigma_w'}+ {\lambda_2 }> 0$, which already appeared in Remark \ref{remred}.

\begin{proposition} \label{noninc}
Let  $w$ a weight and $\sigma_w$ as defined in \eqref{bestrhe}. Let $\lambda_2\ge 0$ and $\lambda_1 {\sigma_w'}+ {\lambda_2 }> 0$. Then there exists $C$ independent of $x$, $r$, and $R$ such that
\begin{equation*}
\frac{r^{\lambda_1}\ w(B(x,r))^{\lambda_2/n}}{R^{\lambda_1}\ w(B(x,R))^{\lambda_2/n}}\le C \qquad \text{for all }R>r.
\end{equation*}
In particular, $\lim_{R\to\infty}R^{\lambda_1}\ w(B(x,R))^{\lambda_2/n}=\infty$.
\end{proposition} 
\begin{proof}
From \eqref{ineqRH} we obtain
\[
  \frac{r^{\lambda_1}\ w(B(x,r))^{\lambda_2/n}}{R^{\lambda_1}\ w(B(x,R))^{\lambda_2/n}}\le C \left(\frac{r}{R}\right)^{\lambda_1 + \frac {\lambda_2 }{\sigma'}},
\]
if $w\in RH_\sigma$. The exponent is positive for some $\sigma<\sigma_w$ due to the hypothesis $\lambda_1 {\sigma_w'}+ {\lambda_2 }> 0$. 
\end{proof}

Lemma \ref{key2} below provides an estimate which is used in the proofs of the next sections. To get it we start with another auxiliary result.
\begin{lemma}\label{keylemma}
 Let $w$ a weight. Let $\theta\ge 0$ such that 
\begin{equation}\label{RHtheta}
 \frac{w(B(0, r))}{ w(B(0, R)) }\le c \left(\frac rR\right)^{n\theta}\qquad \text{for all } 0<r<R.
\end{equation}
Assume that $\lambda_2\ge 0$ and $\lambda_1 + \lambda_2 \theta> 0$. Let $\alpha$ be such that 
\begin{equation} \label{bertiz5}
	0\leq\alpha<
	\begin{cases}
	  \frac{n}{n-\lambda_2 }\left(\lambda_1+ \lambda_2 \theta\right), & \quad\lambda_2<n,\\
	  \infty, & \quad\lambda_2\ge n.
  \end{cases}
\end{equation}
For $1\le p<\infty$, and $f\ge 0$ in $\mathcal L^{p,\lambda}(|x|^\alpha w)$ we have
\begin{equation}\label{basicineq2}
\left(\int_{B(0,r)}f^{p} w\right)^{\frac 1{p}}\le C r^{\left(\lambda_1-\alpha\left[1-\frac{\lambda_2}{n}\right]\right)\frac{1}{p}} w(B(0,r))^{\frac{\lambda_2}{n p}} \|f\|_{\mathcal L^{p,\lambda}(|x|^\alpha w)}.
\end{equation}
The constant depends only on $\alpha, \lambda_1,\lambda_2$ and  $p$, but not on $r$.
\end{lemma}
\begin{remark}
If $w\in RH_\sigma$, then \eqref{RHtheta} is fulfilled for $\theta=1/\sigma'$. Nevertheless, for $w(x)=|x|^\beta$ a direct computation gives $\theta=1+\beta/n$, which allows a larger range of values of $\alpha$ in \eqref{bertiz5} than the one given by the reverse H\"older exponent of the weight in the case $\beta>0$.

On the other hand, notice that if $\lambda_2=0$ or $\lambda_2\ge n$, then the condition in \eqref{bertiz5} does not depend on $\theta$. 
\end{remark}

\begin{proof}[Proof of Lemma \ref{keylemma}]
Let $A_j=B(0, 2^{-j+1}r)\setminus  B(0, 2^{-j}r)$, $j\in \N$. Then
\begin{align*}
\int_{B(0,r)}f^p w&\le C\sum_{j=1}^\infty \int_{A_j}f(y)^p w(y)\left(\frac{|y|}{2^{-j}r}\right)^\alpha dy
\\&\le C\sum_{j=1}^\infty (2^{-j}r)^{\lambda_1-\alpha} \left(\int_{B(0,2^{-j+1}r)} w(y)|y|^\alpha dy\right)^{\frac{\lambda_2}{n}} \|f\|^p_{\mathcal L^{p,\lambda}(|x|^\alpha w)}
\\&\le C\sum_{j=1}^\infty (2^{-j}r)^{\lambda_1-\alpha\left[1-\frac{\lambda_2}{n}\right]} \ w(B(0,2^{-j+1}r))^{\frac{\lambda_2}{n}} \|f\|^p_{\mathcal L^{p,\lambda}(|x|^\alpha w)}
.
\end{align*}
Using \eqref{RHtheta}, to obtain \eqref{basicineq2} it is sufficient to assume $\lambda_1-\alpha\left[1-\frac{\lambda_2}{n}\right]+\lambda_2\theta>0$. 
\end{proof}

\begin{lemma}\label{key2}
Let $w\in A_p$. (a) Let $B=B(x,r)$ be a ball such that $r\le |x|/2$ and $f$ nonnegative. Then for arbitrary $\alpha$ it holds 
\begin{equation}\label{smallball}
\left(\frac 1{|B|}\int_B f\right)^p\le C r^{\lambda_1} (|x|^\alpha w(B))^{\frac{\lambda_2}n-1}\|f\|_{\mathcal L^{p,\lambda}(|x|^\alpha w)}^p.
\end{equation}
(b) Let $B=B(0,r)$ and $f$ nonnegative. Then for $\lambda_2\ge 0$ and $\alpha$ as in \eqref{bertiz5} it holds
\begin{equation}\label{bigball}
\left(\frac 1{|B|}\int_B f\right)^p\le C r^{\lambda_1+\alpha \left(\frac{\lambda_2}n-1\right)}w(B)^{\frac{\lambda_2}n-1}\|f\|_{\mathcal L^{p,\lambda}(|x|^\alpha w)}^p.
\end{equation}
\end{lemma}

\begin{proof}
For $p>1$ we use H\"older's inequality and  the $A_p$  condition \eqref{defap} to write
\begin{equation*}
\left(\frac 1{|B|}\int_B f\right)^p\le \frac 1{|B|^p}\left(\int_B f^p w\right) w^{1-p'}(B)^{p-1}\le C\, \frac 1{w(B)}\int_B f^p w.
\end{equation*}
For $p=1$ the same upper bound is immediate from \eqref{A1}.

(a) Using $|y|\sim |x|$ for $y\in B$  we have
\begin{equation*}
\frac 1{w(B)}\int_B f^p w \le \frac C{|x|^\alpha w(B)}\int_B f^p|y|^\alpha w
\end{equation*}
and \eqref{smallball} follows from the definition of the norm in $\mathcal L^{p,\lambda}(|x|^\alpha w)$.

(b) In this case estimate \eqref{bigball} holds due to \eqref{basicineq2}.
\end{proof}

\subsection{The case of power weights}

When $w(x)=|x|^\beta$ we have some specific results. In particular, we can relate spaces corresponding to different choices of $\lambda_1$ and $\lambda_2$.
 
For balls of type II using that $|y|\sim |x|$ for $y\in B(x,r)$ we have
\begin{equation*}
w(B(x,r))=\int_{B(x,r)}|y|^\beta dy \sim |x|^\beta r^n.
\end{equation*}

Let us define $N_{II}(f; \mathcal L^{p,\lambda}(|x|^\beta))$ as the supremum in \eqref{normdef} restricted to balls of type II, that is, 
\begin{equation*}
\aligned
N_{II}(f; \mathcal L^{p,\lambda}(|x|^\beta))&= \mathop{\sup_{x\ne 0}}_{{0<r\le |x|/4}} \frac 1{r^{\lambda_1}w(B(x,r))^{\lambda_2/n}}\int_{B(x,r)}|f(y)|^p |y|^\beta dy \\
&\sim \mathop{\sup_{x\ne 0}}_{{0<r\le |x|/4}} \frac 1{r^{\lambda_1+\lambda_2}|x|^{\beta \lambda_2/n}}\int_{B(x,r)}|f(y)|^p |y|^\beta dy. 
\endaligned
\end{equation*}

Observe that if $f$ is the characteristic function of any ball of type II, then $N_{II}(f; \mathcal L^{p,\lambda}(|x|^\beta))<\infty$ without restriction on $\lambda$ and $\beta$. Nevertheless, there are values of $\lambda$ and $\beta$ for which $\|f\|_{\mathcal L^{p,\lambda}(|x|^\beta)}<\infty$ implies that $f$ is identically zero (see Proposition \ref{betaspaces} below).

An immediate consequence of the definition is 
\begin{equation}\label{nbi}
N_{II}(f; \mathcal L^{p,\lambda}(|x|^\beta))\sim N_{II}(f; \mathcal L^{p,(\lambda_1+\mu, \lambda_2-\mu)}(|x|^{\beta(1-\mu/n)})).
\end{equation}
This allows to transfer (part of) one of the exponents to the other and when the reduction of the Morrey norms to balls of type II is possible it gives the equivalence of the norms of different spaces. For example, one gets that a certain Komori-Shirai type space is the same as a certain Samko type space, namely,
\begin{equation}\label{samkomori}
\mathcal L^{p,(0,\mu)}(|x|^\beta)=\mathcal L^{p,(\mu,0)}(|x|^{\beta(1-\mu/n)}).
\end{equation}
More generally, we can state the following proposition.

\begin{proposition}\label{betaspaces}
Let $w(x)=|x|^\beta$ for $\beta>-n$. Then the following holds.
\begin{enumerate}[label=(\alph*)]
\item  If $\lambda_1+\lambda_2(1+\beta/n)>0$ and $0< \lambda_1+\lambda_2< n$, then
\begin{equation*}
\|f\|_{\mathcal L^{p,(\lambda_1,\lambda_2)}(|x|^\beta)}\sim \|f\|_{\mathcal L^{p,(\lambda_1+\lambda_2,0)}(|x|^{\beta(1-\lambda_2/n)})}.
\end{equation*}
\item If $\lambda_1+\lambda_2(1+\beta/n)=0$ and $0\le \lambda_1+\lambda_2< n$, then
\begin{equation*}
\|f\|_{\mathcal L^{p,(\lambda_1,\lambda_2)}(|x|^\beta)}\sim \|f\|_{L^{p}(|x|^\beta)}+ \|f\|_{\mathcal L^{p,(\lambda_1+\lambda_2,0)}(|x|^{\beta+\lambda_1+\lambda_2})}.
\end{equation*}
\item If $\lambda_1+\lambda_2(1+\beta/n)<0$, then $\mathcal L^{p,\lambda}(|x|^\beta)=\{0\}$.
\item If $n=\lambda_1+\lambda_2>-\lambda_2\beta/n$, then 
\begin{equation*}
\mathcal L^{p,\lambda}(|x|^\beta)=\{f: |f|^p |x|^{\beta(1-\lambda_2/n)}\in L^\infty\}.
\end{equation*}
\end{enumerate}
\end{proposition}

\begin{proof}
In case (a), the reduction to balls of type II is possible for the involved spaces, as mentioned in Remark \ref{remred}. Therefore, the result follows from \eqref{nbi}.
  
In case (b) we cannot use the reduction to balls of type II for $\mathcal L^{p,(\lambda_1,\lambda_2)}(|x|^\beta)$, and we need to take into account also the balls centered at the origin. But 
\begin{equation*}
r^{\lambda_1} w(B(0,r))^{\lambda_2/n}=c\, r^{\lambda_1+\lambda_2(1+\beta/n)}=c,
\end{equation*}
and the supremum in $r$ over these balls gives a multiple of $\|f\|_{\mathcal L^{p}(|x|^\beta)}$. For balls of type II we obtain the same norm as in part (a) and this is the reason of the last term in (b).

In case (c),  $\lim_{r\to\infty} r^{\lambda_1} w(B(0,r))^{\lambda_2/n}=0$ and the supremum on balls centered at the origin is finite only if the function is identically zero.

The result in (d) corresponds to the power weighted version of Proposition \ref{special}. The proof is the same, only the range of $\lambda_2$ is larger here. We needed $\lambda_2\le n$ to use H\"older's inequality in Proposition \ref{special}, but we can use direct calculation if $w(x)=|x|^\beta$. The condition $\beta \lambda_2>-n^2$ appears beacuse it is needed for the integrability of $|x|^{\beta \lambda_2/n}$.
\end{proof}

The condition $\lambda_1+\lambda_2\in (0,n)$ in (a) and (b) is needed to ensure that the norm of the spaces $\mathcal L^{p,(\lambda_1+\lambda_2,0)}(|x|^{\gamma})$ can be restricted to balls of type II. If $\lambda_1+\lambda_2<0$, we have $\mathcal L^{p,(\lambda_1+\lambda_2,0)}(|x|^{\gamma})=\{0\}$, but  $\mathcal L^{p,(\lambda_1,\lambda_2)}(|x|^{\beta})$ is not trivial: the characteristic functions of balls of type II belong to the space, for example.

We can also give better versions of Lemmas \ref{keylemma} and \ref{key2} in the sense that the range of $\lambda_1$ and $\lambda_2$ to which they apply is larger.
\begin{lemma}\label{keypw}
 Let $1\le p<\infty$, $\beta>-n$, and $\lambda_1 + \lambda_2(1+\beta/n)>0$. Let $f\ge 0$ in $\mathcal L^{p,\lambda}(|x|^{\beta})$ with $\beta>-n$ and $0\le\alpha<\lambda_1+ \lambda_2(1+\frac{\beta}{n})$.
For $r>0$ it holds that
 \begin{equation}\label{basicineq3}
\left(\int_{B(0,r)}f^{p} |y|^{\beta-\alpha} \right)^{\frac 1{p}}\le C \ r^{\left(\lambda_1-\alpha+\lambda_2\left[1+\frac{\beta}{n}\right]\right)\frac{1}{p}}  \|f\|_{\mathcal L^{p,\lambda}(|x|^{\beta})}.
\end{equation}
The constant depends only on the involved parameters, but not on $r$.
\end{lemma}

The proof is similar to that of Lemma \ref{keylemma} with $w(x)=|x|^{\beta-\alpha}$ and we use that the integral of $w(y)|y|^{\alpha}=|y|^{\beta}$ over $B(0,2^{-j+1}r)$ is like $C(2^{-j}r)^{\beta+n}$.

\begin{lemma}\label{key2pw}
Let $1\le p<\infty$. (a) Let $B=B(x,r)$ be a ball such that $r\le |x|/2$ and $f$ nonnegative. Then for arbitrary $\beta$ it holds 
\begin{equation}\label{smallballpw}
\left(\frac 1{|B|}\int_B f\right)^p\le C r^{\lambda_1+\lambda_2-n}|x|^{-\beta(1-\lambda_2/n)}\|f\|_{\mathcal L^{p,\lambda}(|x|^\beta)}^p.
\end{equation}
(b) Let $B=B(0,r)$,  $f$ nonnegative, $\beta>-n$, and $\lambda_1 + \lambda_2(1+\beta/n)>0$. Then for
 \[
\beta\left(1-\frac {\lambda_2}n\right)<\lambda_1+\lambda_2 + n(p-1)
\]  it holds
\begin{equation}\label{bigballpw}
\left(\frac 1{|B|}\int_B f\right)^p\le C r^{\lambda_1+\lambda_2-n-\beta(1-\lambda_2/n)}\|f\|_{\mathcal L^{p,\lambda}(|x|^\beta)}^p.
\end{equation}
\end{lemma}

\begin{proof}
(a) We have 
\begin{align*}
\left(\frac 1{|B|}\int_B f\right)^p \le \frac 1{|B|}\int_B f^p\le C \frac {|x|^{-\beta}}{|B|}\int_B f^p |y|^{\beta}
\end{align*}
and \eqref{smallballpw} follows.

(b) We have 
\begin{align*}
\left(\frac 1{|B|}\int_B f\right)^p & \le \frac 1{|B|}\int_B f^p |y|^{\beta-\alpha} \left(\frac 1{|B|}\int_B  |y|^{(\alpha-\beta)(p'-1)}dy\right)^{p-1} \\
 &\le C r^{\alpha-\beta -n} \int_B f^p |y|^{\beta-\alpha}.
\end{align*}
and \eqref{bigballpw} follows from \eqref{basicineq3} if some conditions are fulfilled. First we need $\alpha > \beta - n(p-1)$ for the integrability of $|y|^{(\alpha-\beta)(p'-1)}$. Moreover, we need $\alpha$ satisfying the conditions of Lemma \ref{keypw}. Both conditions on $\alpha$ are compatible if 
\begin{equation*}
\beta - n(p-1) < \lambda_1+ \lambda_2(1+\frac{\beta}{n})\,,
\end{equation*}
which gives the condition for $\beta$ required in the statement . 
\end{proof}

\section{The Hardy-Littlewood maximal operator}\label{hiru}

In this section we deal with the Hardy-Littlewood maximal operator. First we obtain a necessary condition in the case on nonnegative $\lambda_1$  and $\lambda_2$. It shows that $w$ must belong to a certain Muckenhoupt class of weights.
\begin{proposition}\label{necrange0}
Let $1\le p<\infty$, $\lambda_1\ge 0$, $\lambda_2\ge 0$, and $0<\lambda_1 + \lambda_2< n$. If $M$ is bounded from $\mathcal L^{p,\lambda}(w)$ to $W\mathcal L^{p,\lambda}(w)$, then $w\in A_{\frac{np+\lambda_1}{n-\lambda_2}}$. 
\end{proposition}

\begin{proof}
Let $B$ be a ball of radius $r$. Define $f=\sigma \chi_B$ with $\sigma$ nonnegative to be chosen later. For $x\in B$, we have $Mf(x)\ge \sigma(B)/|B|$. If $t<\sigma(B)/|B|$, then $B= \{x\in B: Mf(x)>t\}$. Assuming that $M$ is bounded from $\mathcal L^{p,\lambda}(w)$ to $W\mathcal L^{p,\lambda}(w)$ we have 
\begin{equation*}
\aligned
& \left(\frac{t^p w(B)}{|B|^{\lambda_1/n}\ w(B)^{\lambda_2/n}}\right)^{1/p}\le C \|\sigma \chi_B\|_{\mathcal L^{p,\lambda}(w)}
\\ & \le  C  \|\sigma \chi_B\|_{L^{\frac {pn}{n-\lambda_1-\lambda_2}}(w^{\frac{n-\lambda_2}{n-\lambda_1-\lambda_2}})}
=C \left(\int_B \sigma^{\frac{pn}{n-\lambda_1-\lambda_2}}w^{\frac{n-\lambda_2}{n-\lambda_1-\lambda_2}}\right)^{\frac {n-\lambda_1-\lambda_2}{pn}},
\endaligned
\end{equation*}
where we used Lemma \ref{emb0.5} in the second  inequality. 
Let $t$ tend to $\sigma(B)/|B|$ and choose $\sigma$ such that $\sigma=\sigma^{\frac{pn}{n-\lambda_1-\lambda_2}}w^{\frac{n-\lambda_2}{n-\lambda_1-\lambda_2}}$, that is, $\sigma=w^{-\frac{n-\lambda_2}{n(p-1)+\lambda_1+\lambda_2}}$. We get 
\begin{equation*}
\frac{w(B)^{1-\frac {\lambda_2} {n}} \sigma(B)^{p-1+\frac {\lambda_1+\lambda_2}{n}}}{|B|^{p+\frac {\lambda_1} {n}}}\le C,
\end{equation*}
with a constant independent of $B$. Therefore, $w\in A_{\frac{np+\lambda_1}{n-\lambda_2}}$. 

Actually, we do not know a priori that $\sigma(B)$ is finite. To get around this problem, we define $\sigma_\epsilon=(w+\epsilon)^{-\frac{n-\lambda_2}{pn-n+\lambda_1+\lambda_2}}$ for $\epsilon>0$ and let $\epsilon$ tend to $0$.
\end{proof}

In this proof we are restricted to the case of nonnegative $\lambda_1$ and $\lambda_2$. The result below for the case $\lambda_1<0$ and $\lambda_2=n$ and power weights shows that the boundedness holds for the weight $|x|^\beta$ and all positive $\beta$. Hence, there is  no possibility to find a $q=q(\lambda_1, \lambda_2,p)$ for which $w\in A_q$ in such case. An interesting question for which we do not have an answer is whether the weight has to be in $A_\infty$ in all the admissible cases, and in particular when $\lambda_2=n$, which corresponds to one of the types of weighted spaces highlighted in the introduction.

Applied to power weights, $w(x)=|x|^\beta$, the condition $w\in A_{\frac{np+\lambda_1}{n-\lambda_2}}$ of Proposition \ref{necrange0} is the same as saying
\begin{equation*}
-n< \beta<\frac{n}{n-\lambda_2}\left[\lambda_1+\lambda_2 + n(p-1)\right].
\end{equation*}
Nevertheless, for these weights we can restrict the bounds for the exponent $\beta$ and even extend the result to some negative values of $\lambda_1$ or $\lambda_2$. This is the content of the next result. 

\begin{proposition}\label{necrangepower2}
Let $1\le p<\infty$, $\beta>-n$,  
and $-\lambda_2 \beta/n< \lambda_1 + \lambda_2 \le n$. 
If $M$ is bounded from $\mathcal L^{p,\lambda}(|x|^\beta)$ to $W\mathcal L^{p,\lambda}(|x|^\beta)$, then 
\begin{equation} \label{rangepowerw1}
\lambda_1+\lambda_2 -n  \le \beta \left(1-\frac {\lambda_2} n\right)<\lambda_1+\lambda_2 + n(p-1).
\end{equation}
\end{proposition}

\begin{proof}
First we look for the condition ensuring that the characteristic function of the unit ball centered at the origin is in $W\mathcal L^{p,\lambda}(|x|^\beta)$. 

Take the ball $B=B(x,r)$ with $r\le |x|/4$. We can assume $|x|<3/2$, otherwise $B$ does not intersect the unit ball. Since $|y|\sim |x|$ for $y\in B$, we need 
\begin{equation*}
\frac {|x|^\beta r^n}{r^{\lambda_1} (|x|^\beta r^n)^{\lambda_2/n}}\le C,
\end{equation*}
for $0<r<|x|/4$. The exponent of $r$ is nonnegative if $\lambda_1+\lambda_2\le n$. When this is satisfied the supremum in $r$ is attained at $r=|x|/4$. Then we require the exponent of $|x|$ to be nonnegative, that is, 
\begin{equation*}
n-\lambda_1-\lambda_2+\beta\left(1-\frac {\lambda_2}n\right)\ge 0.
\end{equation*}

Since $\lambda_1 + \lambda_2 (1+\beta/n)>0$ is the condition for the reduction to balls of type II, the characteristic function of the ball centered at $2$ with radius $1/2$ is in $\mathcal L^{p,\lambda}(|x|^\beta)$. Let $f$ be such function. The maximal operator acting on $f$ satisfies $Mf(x)\ge c$ for some $c>0$ and all $x\in B(0,1)$. Consequently, $Mf\notin W \mathcal L^{p,\lambda}(|x|^\beta)$ for $\beta(n-\lambda_2)< n \left[\lambda_1+\lambda_2 -n\right]$ whenever $\lambda_2<n$. 

The right-hand side condition of \eqref{rangepowerw1} under the assumptions of Proposition \ref{necrange0} has been already proved. But here we give a direct proof valid for more general $\lambda_1$ and $\lambda_2$. 
For 
\begin{align} \label{bertiz3}
\beta \left(1-\frac {\lambda_2}n\right) \ge n(p-1)+\lambda_1+\lambda_2, 
\end{align}
the function $|x|^{-n}\chi_{B(0,1)}$ is in $\mathcal L^{p,\lambda}(|x|^\beta)$ and is not locally integrable. Let $B=B(x,r)$ with $r\le |x|/4$. Since $\int_B |y|^{-np} |y|^\beta dy\sim |x|^{\beta-np} r^n$, we need 
\begin{equation*}
\frac {|x|^{\beta-np} r^n}{r^{\lambda_1} (|x|^\beta r^n)^{\lambda_2/n}}\le C.
\end{equation*}
The exponent of $r$ is nonnegative for  $\lambda_1+\lambda_2\le n$, and assuming this condition and replacing $r$ with $|x|/4$, the resulting exponent of $x$ is nonnegative if \eqref{bertiz3} holds.

Notice that \eqref{bertiz3} together with the assumption $\lambda_1 + \lambda_2+\lambda_2 \beta/n>0$ implies $\beta>n(p-1)$. This condition is natural because otherwise $|x|^{-n}\chi_{B(0,1)}\notin L^p(B(0,\delta); |x|^\beta)$.
\end{proof}

\begin{remark}
The condition $\lambda_1 + \lambda_2 (1+\beta/n)>0$ also imposes a bound to $\beta$, but this is of different nature. As can be seen in Proposition \ref{betaspaces}, when $\lambda_1 + \lambda_2 (1+\beta/n)<0$ the Morrey space is trivial.

Condition \eqref{rangepowerw1} for $\lambda_2=n$  is the same as $-np<\lambda_1\le 0$.
\end{remark}

We study now the boundedness of the Hardy-Littlewood maximal operator in Morrey spaces with products of $A_p$ weights and power weights. In particular, we obtain sharp results for power weights. We will be under the assumptions that make possible to consider only balls of type II (see Proposition \ref{reduct}). We also use the following. Let  $B$ be a ball and let $g$ be a nonnegative function supported outside $2B$. There exists a ball $B'$ containing $B$ such that 
\begin{equation}\label{redmax}
Mg(x)\sim \frac 1{|B'|}\int_{B'} g\quad \text{for all } x\in B.
\end{equation}
Clearly the right-hand side is smaller than the left-hand side. The opposite inequality holds with a dimensional constant.

\begin{theorem}\label{sufHL}
Let $1<p<\infty$ and $w\in A_p$. Assume that $0\le \lambda_2\le n$ and $\lambda_1\sigma_w'+\lambda_2> 0$. The Hardy-Littlewood maximal operator is bounded on $\mathcal L^{p,\lambda}(|x|^\alpha w)$ for $\lambda_1\sigma_w'+\lambda_2<n$ and $\alpha$ satisfying the conditions of Lemma \ref{keylemma}. If $p=1$, the same result holds with a weak-type estimate.
\end{theorem}

\begin{proof}
We can limit ourselves to consider balls of the form $B:=B(x,r)$ with $r\le |x|/4$ (Remark \ref{remred}). For a given nonnegative $f$ we write $f=f_1+f_2$, where $f_1=f\chi_{2B}$ and $f_2=f\chi_{(2B)^c}$. Then $Mf(y)\le Mf_1(y)+Mf_2(y)$ and we can work with each one of both terms separately. We denote $u(B)=r^{\lambda_1} w(B)^{\lambda_2/n}$.

For the first term we have
\begin{equation*}
\frac 1{u(B)}\int_B Mf_1(y)^p |y|^\alpha w(y)dy\le C \frac {|x|^\alpha}{u(B)}\int_{\R^n} f_1(y)^p w(y)dy,
\end{equation*}
where we used the fact that $|y|\sim |x|$ for $y\in B$ and the boundedness of $M$ on $L^p(w)$ for $w\in A_p$. Using the support condition of $f_1$ we have \begin{equation*}
|x|^\alpha \int_{\R^n} f_1(y)^p w(y)dy\le C \int_{2B} f(y)^p |y|^\alpha w(y)dy\le u(2B) \|f\|_{\mathcal L^{p,\lambda}(|x|^\alpha w)}.
\end{equation*}
The doubling property of $u$ leads to the desired estimate. 

To deal with $Mf_2(y)$ we take into account the observation previous to the theorem and use \eqref{redmax} with $g=f_2$. Then it is enough to estimate 
\begin{equation*}
\frac 1{u(B)}\int_B \left(\frac 1{|B'|}\int_{B'} f_2\right)^p |y|^\alpha w(y)dy,
\end{equation*}
where $B\subset B'$. This is in turn bounded by a constant times  
\begin{equation}\label{bound1}
\frac {|x|^\alpha w(B)}{u(B)}\left(\frac 1{|B'|}\int_{B'} f\right)^p.
\end{equation}
We distinguish two type of balls $B'$ depending of the relation between the center $x_{B'}$ and the radius $r_{B'}$. If $r_{B'}\le |x_{B'}|/2$, then $|y|\sim |x|$ for $y\in B'$. If $r_{B'}>|x_{B'}|/2$, then we can replace $B'$ with a ball centered at the origin and comparable radius paying with a dimensional constant. Hence we can assume that $B'$ is of the form $B(0,r_{B'})$ in the latter case. We use the estimates of Lemma \ref{key2}.

In the first case we replace $u(B)=r^{\lambda_1}\, (|x|^\alpha\,w(B))^{\lambda_2/n}$ and use \eqref{smallball} to bound \eqref{bound1} with
\begin{equation*}
\aligned
&C \frac {|x|^\alpha w(B)}{r^{\lambda_1}\, (|x|^\alpha\,w(B))^{\lambda_2/n}}\,r_{B'}^{\lambda_1} \, (|x|^\alpha w(B'))^{\frac{\lambda_2}n-1}\|f\|_{\mathcal L^{p,\lambda}(|x|^\alpha w)}^p \\
&\qquad\le C \left(\frac {r}{r_{B'}}\right)^{-\lambda_1} \left(\frac {w(B)}{w(B')}\right)^{1-\frac{\lambda_2}n}\|f\|_{\mathcal L^{p,\lambda}(|x|^\alpha w)}\\
&\qquad\le C \left(\frac {r}{r_{B'}}\right)^{-\lambda_1+\left(1-\frac{\lambda_2}n\right)\frac n{\sigma'}}\|f\|_{\mathcal L^{p,\lambda}(|x|^\alpha w)}^p,
\endaligned
\end{equation*}
where $\sigma<\sigma_w$ (that is, $w\in RH_\sigma$). We get the boundedness when the exponent is positive because $r<r_{B'}$. This gives the condition $\lambda_1\sigma_w'+\lambda_2<n$.

If $B'$ is a ball centered at the origin, we use \eqref{bigball} to bound the right-hand side of \eqref{bound1} with
\begin{equation*}
C \frac {|x|^\alpha w(B)}{r^{\lambda_1}\, (|x|^\alpha\,w(B))^{\lambda_2/n}}\,r_{B'}^{\lambda_1-\alpha\left(1-\frac{\lambda_2}{n}\right)} w(B')^{\frac{\lambda_2}n-1} \|f\|_{\mathcal L^{p,\lambda}(|x|^\alpha w)}.
\end{equation*}
What we need is
\begin{equation}\label{bound2}
\left(\frac {r_{B'}}{r}\right)^{\lambda_1} \left(\frac {|x|}{r_{B'}}\right)^{\alpha\left(1-\frac{\lambda_2}{n}\right)} \left(\frac {w(B)}{w(B')}\right)^{1-\frac{\lambda_2}{n}} \le C.
\end{equation}
The term in the middle is bounded by $1$ for $\lambda_2\le n$ because $|x|<r_{B'}$. Then we are in the same situation as in the previous case. 

When $p=1$ we need to use the weak $(1,1)$ inequality of $M$ with $A_1$ weights to deal with $Mf_1$. For $Mf_2$ the strong estimate above is valid for $w\in A_1$.
\end{proof}

Let us summarize the results for the special cases of weighted Morrey spaces mentioned in the introduction:
\begin{itemize}
\item in the case $\lambda_1=0$ and $0<\lambda_2<n$ the reverse H\"older condition does not play any role and the condition on $\alpha$ is given by $0\le \alpha < \frac {n\lambda_2\theta}{n-\lambda_2}$;
\item  in the case $0<\lambda_1<n$ and $\lambda_2=0$ we need $n-\lambda_1 \sigma_w'> 0$ and $\alpha<\lambda_1$ (this is the result in \cite{DR19});
\item in the case $\lambda_1<0$ and $\lambda_2=n$, we need $n+\lambda_1 \sigma_w'> 0$ and $\alpha$ can take any positive value.
\end{itemize}

In the case of power weights $|x|^\beta$ we can give the precise range even with an endpoint estimate for negative $\beta$. The condition $\lambda_1\sigma_w'+\lambda_2> 0$ of the previous theorem for positive $\beta$ is $\lambda_1+\lambda_2> 0$. We can extend it to $\lambda_1 + \lambda_2 (1+\beta/n)>0$ by using the result in Lemma \ref{keypw}.

\begin{theorem}\label{pwmaxhl}
 Let $1<p<\infty$, $\beta>-n$, and $-\lambda_2 \beta/n< \lambda_1 + \lambda_2 < n$. The Hardy-Littlewood maximal operator is bounded on $\mathcal L^{p,\lambda}(|x|^\beta)$ if and only if $\beta$ satisfies the restrictions \eqref{rangepowerw1}. For $p=1$ they hold as weak-type estimates.
 \end{theorem}
 
\begin{proof}
 The necessity has been proved in Proposition \ref{necrangepower2}. For the sufficiency we proceed as in the previous theorem and consider $Mf_1$ and $Mf_2$. 
 
To deal with $Mf_1$ we notice that
\begin{equation*}
\int_B Mf_1(y)^p |y|^\beta dy\le C |x|^\beta\int_{\R^n} (Mf_1)^p\le C |x|^\beta\int_{2B} f^p \le C \int_{2B} f^p|y|^\beta,
\end{equation*}
and the estimate follows without any condition on $\beta$.

On the other hand, with $B'$ as in the previous theorem, we have
\begin{equation*}
\int_B Mf_2(y)^p |y|^\beta dy\le C |x|^\beta r^n \left(\frac 1{|B'|}\int_{B'} f\right)^p
\end{equation*}
and we apply the estimates of Lemma \ref{key2pw}. If $B'$ is such that $r_{B'}\le |x_{B'}|/2$, we have $|x_B'|\sim |x|$ and using \eqref{smallballpw} the estimate holds. If $B'$ is centered at the origin we use \eqref{bigballpw} for which we need the right-hand side condition of \eqref{rangepowerw1}. If this is granted we have
\begin{equation*}
\aligned
\frac 1{r^{\lambda_1}(|x|^\beta r^n)^{\lambda_2/n}}&\int_B Mf_2(y)^p |y|^\beta dy\\ &\le C \left(\frac {|x|}{r_{B'}}\right)^{\beta\left(1-\frac{\lambda_2}{n}\right)} \left(\frac {r}{r_{B'}}\right)^{n-\lambda_1-\lambda_2} \|f\|^p_{\mathcal L^{p,\lambda}(|x|^\beta)}.
\endaligned
\end{equation*}
We use the condition $\lambda_1+\lambda_2\le n$ to replace $r$ with $|x|$ and the needed estimate holds for the left-hand side condition of \eqref{rangepowerw1} because $|x|<r_{B'}$.  
\end{proof}

The results obtained in Theorem \ref{pwmaxhl} applied to the special cases of power-weighted Morrey spaces are the following:
\begin{itemize}
\item in the case $\lambda_1=0$ and $0<\lambda_2<n$, $M$ is bounded on $\mathcal L^{p,\lambda}(|x|^\beta)$ if and only if $-n<\beta< \frac{n}{n-\lambda_2}\left[\lambda_2 + n(p-1)\right]$;
\item  in the case $0<\lambda_1<n$ and $\lambda_2=0$, $M$ is bounded on $\mathcal L^{p,\lambda}(|x|^\beta)$ if and only if  $-n+\lambda_1\le \beta < n(p-1)+\lambda_1$ (this is the result in \cite{DR19});
\item in the case $\lambda_2=n$ and $-np<\lambda_1<0$, $M$ is bounded on $\mathcal L^{p,\lambda}(|x|^\beta)$ if $-n-\lambda_1< \beta$. For $\beta< -n-\lambda_1$ the corresponding space is $\{0\}$ according to part (c) of Proposition \ref{betaspaces}. For $-n-\lambda_1= \beta$ and $\lambda_1>-n$ we are in part (b) of the same proposition and the result follows  from the boundedness of $M$ on $L^p(|x|^\beta)$ and on  $\mathcal L^{p,(\lambda_1+n,0)}(1)$.  The condition $0<\lambda_1+np$ is necessary according to the right-hand side of \eqref{rangepowerw1}. 
\item In all the described cases the result holds for $p=1$ with the weak-type estimate.
\end{itemize}

An interesting unsolved question is the characterization of the class of weights for which $M$ is bounded on $\mathcal L^{p,\lambda}(w)$. This seems to be unknown even for the spaces of Komori-Shirai type and of Samko type. See \cite{Ta15} for some partial results. We would like to point out that \cite[Theorem 1.1]{WZC17} claims that for spaces of Komori-Shirai type the class $A_{\frac{np}{n-\lambda_2}}$ (recall that $\lambda_1=0$) is necessary and sufficient. But the sufficiency is not proved in the paper and is attributed to \cite{KS09}, which is not accurate because the result in \cite{KS09} only gives the sufficiency of $w\in A_p$. In \cite{NST19} a characterization is obtained for the boundedness on weighted local Morrey spaces, which are defined as the usual Morrey spaces but considering only balls centered at the origin.

\section{Extrapolation techniques}\label{bost}

In this section we use extrapolation techniques as in \cite{DR18, DR19} to obtain abstract results in the class of weighted Morrey spaces defined in this paper. They can subsequently be applied to a variety of operators, their vector-valued extensions, weak-type estimates, etc. 

\begin{theorem}\label{teo1K}
Let $1\le p_0<\infty$ and let $\mathcal F$ be a collection of nonnegative measurable pairs of functions. Assume that for every $(f,g)\in \mathcal F$ and every $u\in A_{p_0}$ we have
\begin{equation}\label{hyp2}
\|g\|_{L^{p_0}(u)}\le C \|f\|_{L^{p_0}(u)},
\end{equation}
where $C$ does not depend on the pair $(f,g)$ and it depends on $u$ only in terms of $[u]_{A_{p_0}}$. 
Then for $1<p<\infty$ and $w\in A_p$ it holds
 \begin{equation}\label{boundmorrey}
\|g\|_{\mathcal L^{p,\lambda}(|x|^\alpha w)}\le C \|f\|_{\mathcal L^{p,\lambda}(|x|^\alpha w)},
\end{equation}
for $0\le\lambda_2\le n$, $0<\lambda_1 {\sigma_w'}+ \lambda_2 <n$, and 
$\alpha$ as in \eqref{bertiz5}. If \eqref{hyp2} holds for $p_0=1$, then the conclusion \eqref{boundmorrey} is satisfied for $p=1$.
\end{theorem}

\begin{remark}
Notice that the conditions on $\lambda_1$ and $\lambda_2$ exclude the possibility of $\lambda_1+ \lambda_2=n$.  This was to be expected because otherwise, according to Proposition \ref{special}, we would be extrapolating to a space of $L^\infty$ type and this is not possible. 
\end{remark}

\begin{proof}
 First we assume $p_0=1$, that is, $\eqref{hyp2}$ with $u\in A_1$. 
 
 Let $p>1$ and $w\in A_p$. Hence $w^{1-p'}\in A_{p'}$ and we can choose $s>1$ such that $w^{1-p'}\in A_{p'/s}$. Let $B=B(x,r)$ be a ball of type II. Then we need to estimate 
 \begin{equation}\label{laubat}
\left(\frac 1{r^{\lambda_1}(|x|^\alpha w(B))^{\frac {\lambda_2}n}}\int_{B} g^p |y|^\alpha w\right)^{\frac 1p}\sim
\frac {|x|^{\alpha\left(1-\frac {\lambda_2}n\right)\frac 1p}}{r^{\frac {\lambda_1}p}\,w(B)^{\frac {\lambda_2}{np}}}\left(\int_{B} g^p w\right)^{\frac 1p}.
\end{equation}
By duality there exists $h$ nonnegative such that 
\begin{equation*}
\int_B h^{p'}w=1 \quad \text{and} \quad \left(\int_{B} g^p w\right)^{\frac 1p}=\int_{B} g h w.
\end{equation*}
Since $hw\chi_B\le M(h^sw^s\chi_B)^{1/s}$ and $M(h^sw^s\chi_B)^{1/s}$ is an $A_1$ weight for $s>1$, we can write
\begin{equation*}
\int_{B} g h w\le \int_{\rn} g M(h^sw^s\chi_B)^{1/s}\le C \int_{\rn} f M(h^sw^s\chi_B)^{1/s},
\end{equation*}
where in the second inequality we use the hypothesis. To ensure that $M(h^sw^s\chi_B)^{1/s}$ is in $A_1$ we need to check that it is finite almost everywhere and for this it suffices to show that $h^sw^s\chi_B$ is integrable. We prove this and get a bound for future use. We have
\begin{align} 
\begin{split} \label{p3}
	 \left(\int_{B} h^s w^{s-1} w\right)^\frac{1}{s}
	&\le\left(\int_{B} h^{p'} w\right)^\frac{1}{p'}
	\left(\int_{B} w^{\frac{s(p'-1)}{p'-s}} \right)^{\frac1{s}-\frac{1}{p'}}\\ 
	& \le c\  |B|^{\frac{1}{s}}w^{1- p'}(B)^{-\frac{1}{p'}}
		\le c\ w(B)^{\frac{1}{p}}r^{-\frac{n}{s'}},
\end{split} 
\end{align}
where the second inequality holds because $w^{1-p'}\in A_{p'/s}$ (the exponent of $w$ in the integral can be written as $(1-p')/(1-(p'/s)$) and in the last one we use 
\begin{equation*}
c_nr^n=|B|\le w(B)^{\frac{1}{p}} w^{1-p'}(B)^{\frac{1}{p'}}.
\end{equation*}  

We decompose the integral of $f M(h^sw^s\chi_B)^{1/s}$ over $\rn$  in the form
\begin{equation}\label{bitan}
 \int_{2B} f M(h^sw^s\chi_B)^{1/s} + \sum_{j=1}^\infty \int_{2^{j+1}B\setminus 2^{j}B} f M(h^sw^s\chi_B)^{1/s}.
\end{equation}

To deal with the first term in \eqref{bitan} we use H\"older's inequality,
\begin{equation*}
\int_{2B} f M(h^sw^s\chi_B)^{1/s}\le \left(\int_{2B} f^p w\right)^{\frac 1p} \left(\int_{2B} M(h^sw^s\chi_B)^{p'/s} w^{1-p'}\right)^{\frac 1{p'}}.
\end{equation*}
Since $s$ has been chosen such that $w^{1-p'}\in A_{p'/s}$, we use the weighted boundedness of $M$ to see that the last term is less than a constant. On the other hand,
\begin{equation*}
\aligned
 \left(\int_{2B} f^p w\right)^{\frac 1p} &\le  C |x|^{-\alpha/p}\left(\int_{2B} f^p |y|^\alpha w\right)^{\frac 1p}\\
& \le 
C |x|^{-\alpha/p}{(2r)^{\lambda_1/p}(|x|^\alpha w(2B))^{\lambda_2/np}} \|f\|_{\mathcal L^{p,\lambda}(|x|^\alpha w)}.
\endaligned
\end{equation*}
Together with \eqref{laubat} and the doubling property of $w$ we obtain here the desired bound. 

To deal with the integral on the annulus $2^{j+1}B\setminus 2^{j}B$ we first realize that for a function supported on $B$ the maximal function at a point $y$ far from $B$ only needs to take into account averages on balls containing $y$ and intersecting $B$. Therefore, for $y\in 2^{j+1}B\setminus 2^{j}B$, such balls have radius at least $2^{j-1}r$ and up to a dimensional constant the maximal function is comparable to the average on $2^{j+1}B$. Thus the value of $M(h^sw^s\chi_B)(y)$ for $y\in 2^{j+1}B\setminus 2^{j}B$ is equivalent to 
\begin{equation}\label{maxfar}
\frac 1{(2^{j}r)^n} \int_B h^sw^s.
\end{equation}

We use the bound \eqref{p3} for the integral and write 
\begin{equation*}
\sum_{j=1}^\infty \int_{2^{j+1}B\setminus 2^{j}B} f M(h^sw^s\chi_B)^{1/s}\le C\sum_{j=1}^\infty 2^{\frac{jn}{s'}}\left(\frac 1{|2^{j+1}B|}{\int_{2^{j+1}B}} f \right) w(B)^{\frac{1}{p}}.
\end{equation*}  
Let $j_0$ be the largest $j$ for which $2^{j+1}r\le |x|/2$. We can apply inequality \eqref{smallball} to the balls $2^{j+1}B$ for $j\le j_0$. Then
\begin{equation*}
\aligned
& 2^{\frac{jn}{s'}}\left(\frac 1{|2^{j+1}B|}{\int_{2^{j+1}B}} f\right) w(B)^{\frac{1}{p}}\\
&\qquad \le C 2^{\frac{jn}{s'}} [(2^j r)^{\lambda_1} (|x|^\alpha w(2^{j+1}B))^{\frac{\lambda_2}n-1}]^{\frac 1p}w(B)^{\frac{1}{p}}\|f\|_{\mathcal L^{p,\lambda}(|x|^\alpha w)}\\
& \qquad\le C 2^{\frac{jn}{s'}} (2^j r)^{\frac {\lambda_1}p} |x|^{\alpha\left(\frac{\lambda_2}n-1\right)\frac 1p}2^{jn\left(\frac{\lambda_2}n-1\right)\frac 1{p\sigma'}} w(B)^{\frac{\lambda_2}{np}}\|f\|_{\mathcal L^{p,\lambda}(|x|^\alpha w)},
\endaligned
\end{equation*}  
for $\sigma< \sigma_w$, using in the second inequality the bound given by \eqref{ineqRH} for  $w\in RH_\sigma$ and $\lambda_2<n$.

In the case $j>j_0$ we have $2^{j+1}B=B(x, 2^{j+1}r)\subset B(0,2^{j+3}r)$ and we can replace the average of $f$ over $2^{j+1}B$ with the average over $B(0,2^{j+3}r)$. Using now \eqref{bigball} we get
\begin{equation*}
\aligned
&2^{\frac{jn}{s'}}\left(\frac 1{|B(0,2^{j+3}r)|}{\int_{B(0,2^{j+3}r)}} f \right) w(B)^{\frac{1}{p}}\\
&\qquad\le C 2^{\frac{jn}{s'}}[(2^jr)^{\lambda_1+\alpha \left(\frac{\lambda_2}n-1\right)}w(B(0,2^{j+3}r))^{\frac{\lambda_2}n-1}]^{\frac 1p}w(B)^{\frac{1}{p}}\|f\|_{\mathcal L^{p,\lambda}(|x|^\alpha w)}\\
&\qquad\le C 2^{\frac{jn}{s'}} (2^j r)^{\frac {\lambda_1}p} |x|^{\alpha\left(\frac{\lambda_2}n-1\right)\frac 1p}2^{jn\left(\frac{\lambda_2}n-1\right)\frac 1{p\sigma'}} w(B)^{\frac{\lambda_2}{np}}\|f\|_{\mathcal L^{p,\lambda}(|x|^\alpha w)},
\endaligned
\end{equation*} 
using the fact that $|x|\le 2^{j+2}r$ and the bound given by $w\in RH_\sigma$, taking into account in both cases that $\lambda_2<n$. 

Using \eqref{laubat} we obtain the needed bound if 
\begin{equation}\label{sumgeo}
\sum_{j=1}^{\infty} 2^{j\left(\frac{n}{s'}+\frac{\lambda_1}p+ \frac{\lambda_2-n}{p\sigma'}\right)}\le C.
\end{equation} 
The condition $\lambda_1 {\sigma_w'}+ {\lambda_2 -n}<0$ allows us to choose $\sigma$ close to $\sigma_w'$ and  $s$ close enough to $1$ such that the exponent of $2^j$ in the series is negative. This ends the proof for $p_0=1$ and $p>1$.

For $p_0=1$ and $p=1$ the situation is similar, but easier, because duality is not needed. First we have \eqref{laubat} with $p=1$. For $w\in A_1$, choose $s>1$ such that $w^s\in A_1$. This implies $Mw^s(x)\le C w^s(x)$ a.e. and also
\begin{equation}\label{lauzazpi}
\left(\frac{w^s(B)}{|B|}\right)^{\frac 1s}\le C \frac{w(B)}{|B|}.
\end{equation}
Notice that this is the same as $w\in RH_s$. Hence any $s\in (1,\sigma_w)$ can be chosen. Proceeding as before, 
\begin{equation*}
\aligned
\int_{B} g w& \le \int_{\rn} g M(w^s\chi_B)^{1/s}\le C \int_{\rn} f M(w^s\chi_B)^{1/s}\\
 & \le C \int_{2B} f M(w^s\chi_B)^{1/s} + C \sum_{j=1}^\infty \int_{2^{j+1}B\setminus 2^{j}B} f M(w^s\chi_B)^{1/s}.
 \endaligned
\end{equation*}
In the first integral we use $M(w^s\chi_B)^{1/s}\le M(w^s)^{1/s}\le Cw$ and
\begin{equation*}
\int_{2B} f w\le 
C |x|^{-\alpha}{(2r)^{\lambda_1}(|x|^\alpha w(2B))^{\lambda_2/n}} \|f\|_{\mathcal L^{1,\lambda}(|x|^\alpha w)}.
\end{equation*}

As before, the value of $M(w^s\chi_B)(y)$ at $y\in 2^{j+1}B\setminus 2^{j}B$ is comparable to \eqref{maxfar} with $h\equiv 1$. Since $w\in A_1$ we also have
\begin{equation*}
\frac{w(2^{j+1}B)}{|2^{j+1}B|}\le C w(y) \quad \text{a.e. } y\in 2^{j+1}B.
\end{equation*}
Therefore we can write
\begin{equation*}
\aligned
\int_{2^{j+1}B\setminus 2^{j}B} f M(w^s\chi_B)^{1/s}&\le C\left(\int_{2^{j+1}B} f w\right) \frac{|2^{j+1}B|}{w(2^{j+1}B)}\ \left(\frac{w^s(B)}{|2^{j+1}B|}\right)^{\frac 1s}\\
&\le C\left(\int_{2^{j+1}B} f w\right) 2^{\frac {jn}{s'}} \frac{w(B)}{w(2^{j+1}B)},
\endaligned
\end{equation*}
using \eqref{lauzazpi}. We define $j_0$ as in the proof of the case $p>1$ and in the balls with $j\le j_0$ we use that $|y|\sim |x|$ to insert $|y|^\alpha$ inside the integral. For $j>j_0$ we replace $2^{j+1}B$ with a ball centered at the origin and use \eqref{basicineq2} with $p=1$. The details are left to the reader.

Assume now that \eqref{hyp2} holds for some $p_0>1$. By the well-known extrapolation for weighted Lebesgue spaces we know that it holds for any $p_0>1$, in particular, 
\begin{equation}\label{hyp22}
\|g^{p_0}\|_{L^{1}(u)}\le C \|f^{p_0}\|_{L^{1}(u)} \quad \text{for all}\ u\in A_1.
\end{equation}

Given $p>1$ and $w\in A_p$, choose $p_0\in (1,p)$ such that $w\in A_{p/p_0}$. For such $p_0$ we have \eqref{hyp22} and we can apply the first part of the proof to the pair $(f^{p_0}, g^{p_0})$ to deduce that for the given $w$ it holds 
\begin{equation*}
\|g^{p_0}\|_{{\mathcal L^{\frac p{p_0},\lambda}(|x|^\alpha w)}}\le C \|f^{p_0}\|_{{\mathcal L^{\frac p{p_0},\lambda}(|x|^\alpha w)}},
\end{equation*}
which is the desired inequality.
\end{proof}

As for the Hardy-Littlewood maximal operator we have a theorem for power weights in which the range of $\lambda_1$ and $\lambda_2$ is larger. 

\begin{theorem}\label{teo2K}
Let $\lambda_1 + \lambda_2<n$. Let $\beta>-n$ be such that
\begin{equation*}
\max(0, \beta-n(p-1))<\lambda_1 + \lambda_2 (1+\beta/n)<n+\beta.
\end{equation*}
Assume that for a collection of pairs of nonnegative functions the hypotheses of Theorem \ref{teo1K} hold. Then for $1<p<\infty$ it holds
 \begin{equation}\label{boundmorreypw}
\|g\|_{\mathcal L^{p,\lambda}(|x|^\beta)}\le C \|f\|_{\mathcal L^{p,\lambda}(|x|^\beta)}.
\end{equation}
If \eqref{hyp2} holds for $p_0=1$, then the conclusion \eqref{boundmorreypw} is satisfied for $p=1$.
\end{theorem}

\begin{proof}
 We take $w(x)=|x|^\gamma$ for $-n<\gamma<n(p-1)$ and run the proof as in the preceding theorem. Let $\beta=\gamma+\alpha$. Take into account that $w(B)\sim |x|^\gamma r^n$.
 
For the first estimate in the proof, the integral over $2B$, there is no restriction on $\alpha$, that is, no restriction on $\beta$. 
 
 In the second part we have to integrate $f$ over $2^{j+1}B$ and we need to use the estimates of Lemma \ref{key2pw}. When $j\le j_0$, again without condition on $\beta$, from \eqref{smallballpw} we have 
 \begin{equation*}
\frac 1{|2^{j+1}B|}\int_{2^{j+1}B} f\le C (2^{j+1}r)^{\frac{\lambda_1+\lambda_2-n}p}|x|^{-\beta(1-\frac {\lambda_2}{n})\frac 1p}\|f\|_{\mathcal L^{p,\lambda}(|x|^\beta)}.
\end{equation*}
Multiplied by $2^{\frac{jn}{s'}}$ and the factor in front of the integral in \eqref{laubat} we see that if $\lambda_1+\lambda_2-n<0$, we can choose $s$ close to 1 such that the exponent of $2^j$ in the sum is negative. 

When  $j> j_0$ we integrate on a ball centered at the origin as in the previous proof and use \eqref{bigballpw}. Including the factor in \eqref{laubat} we have 
\begin{equation*}
\aligned
&\sum_{j=j_0}^\infty 2^{\frac{jn}{s'}}\frac {|x|^{\alpha\left(1-\frac {\lambda_2}n\right)\frac 1p}}{r^{\frac {\lambda_1}p}\,w(B)^{\frac {\lambda_2}{np}}} \left(\frac 1{|B(0,2^{j+3}r)|}{\int_{B(0,2^{j+3}r)}} f \right) w(B)^{\frac{1}{p}}\\
&\qquad\le C\sum_{j=j_0}^\infty  2^{\frac{jn}{s'}} \frac {|x|^{(\alpha+\gamma)\left(1-\frac {\lambda_2}n\right)\frac 1p}} {r^{\frac{\lambda_1+\lambda_2-n}p}} (2^jr)^{[\lambda_1+\lambda_2-n-\beta(1-\frac {\lambda_2}{n})]\frac 1p}\|f\|_{\mathcal L^{p,\lambda}(|x|^\alpha w)}\\
&\qquad\le C  2^{j_0[\frac{n}{s'}+\lambda_1+\lambda_2-n]} {|x|^{\beta\left(1-\frac {\lambda_2}n\right)\frac 1p}} (2^{j_0}r)^{-\beta(1-\frac {\lambda_2}{n})\frac 1p}\|f\|_{\mathcal L^{p,\lambda}(|x|^\alpha w)},\endaligned
\end{equation*} 
where in the last step we sum the geometric series if the exponent of $2^j$ is negative. This is possible for  $\lambda_1 + \lambda_2 (1+\beta/n)<n+\beta$ by taking $s$ close enough to 1. 
By definition of $j_0$ we have that $|x|\sim 2^{j_0+2}r$. Thus we are left with a power of $2^{j_0}$. The exponent $\lambda_1 + \lambda_2-n+n/s'$ is negative for $s$ close to $1$ because $\lambda_1+\lambda_2-n<0$.
\end{proof}

For the special cases of weighted Morrey spaces described in the introduction the range of power weights obtained for $p>1$ is the same as for the Hardy-Littlewood maximal operator, except the left endpoint (see description after Theorem \ref{pwmaxhl}). The description for $p=1$ is similar, but we need \eqref{hyp2} with $p_0=1$. 

There are other versions of the extrapolation theorem. In particular, it is impossible for operators that are not bounded on $L^p(\rn)$ in the full range $1<p<\infty$ to satisfy \eqref{hyp2}. Nevertheless, in some cases the so-called limited range extrapolation provides a substitute. The proof of the previous theorems in this section can be adapted to cover such case.

\begin{theorem}\label{teolim}
Let $0<p_-\le p_0<p_+\le \infty$. Let $\mathcal F$ be a collection of nonnegative measurable pairs of functions. Assume that for every $(f,g)\in \mathcal F$ and every $w\in A_{\frac{p_0}{p_-}}\cap RH_{\left(\frac{p_+}{p_0}\right)'}$ we have
\begin{equation}\label{hyp3}
\|g\|_{L^{p_0}(w)}\le C \|f\|_{L^{p_0}(w)},
\end{equation}
where $C$ does not depend on the pair $(f,g)$ and it depends on $w$ only in terms of the $A_{\frac{p_0}{p_-}}$ and $RH_{\left(\frac{p_+}{p_0}\right)'}$ constants of $w$. Let $0\le\lambda_2\le n$ and $0<\lambda_1 {\sigma_w'}+ \lambda_2$. Then if $p_-<p<p_+$ and $w\in A_{\frac{p}{p_-}}$ with $\sigma_w> {\left(\frac{p_+}{p}\right)'}$, it holds that 
 \begin{equation*} 
\|g\|_{\mathcal L^{p,\lambda}(|x|^\alpha w)}\le C \|f\|_{\mathcal L^{p,\lambda}(|x|^\alpha w)},
\end{equation*}
for  $\alpha$ as in \eqref{bertiz5} and 
\begin{equation*}
{\lambda_1}\sigma'_w+{\lambda_2}<n\left(1-\frac{p}{p_+}{\sigma'_w}\right).
\end{equation*}

Moreover, if the hypothesis holds for $p_0=p_-$, then the results are valid for $p=p_-$.
\end{theorem}

\begin{proof}
The proof can be reduced to considering the case $p_-=1$ and $p_+=b$, where $b\in (1,\infty]$. The case $b=\infty$ corresponds to Theorem \ref{teo1K}. The proof in the case $b<\infty$ is the same as for Theorem \ref{teo1K}, the only difference being that for $p_0=1$ one needs to choose the weight $M(h^sw^s\chi_B)^{1/s}\in A_1\cap RH_{b'}$, that is, $s>b'$. This has to be compatible with the condition $w^{1-p'}\in A_{p'/s}$ also required in the proof. This is possible for $w\in A_p\cap RH_{\left(\frac bp\right)'}$. Indeed, this implies $w^{(b/p)'}\in A_{(b/p)'(p-1)+1}$, which is equivalent to $w^{1-p'}\in A_{p'/b'}$. Hence, there exists $s>b'$ for which $w^{1-p'}\in A_{p'/s}$ and the choice of $s$ is possible. Once we have checked such conditions, the proof follows in the same way and we require the convergence of the geometric series \eqref{sumgeo}. Since $s$ can be taken as close to $b'$ and $\sigma$ as close to $\sigma'_w$ as needed, the condition given in the statement suffices. 

For general $p_-$, one has to scale the exponents defining $b=p_+/p_-$ and replacing $p$ by $p/p_-$.  
\end{proof}

For the specific case of power weights the result corresponding to Theorem \ref{teo2K} under the restricted assumptions of Theorem \ref{teolim} is the following. 

\begin{theorem}\label{teo2Klim}
Let $0<p_-\le p_0<p_+\le \infty$. Let $\mathcal F$ be a collection of nonnegative measurable pairs of functions satisfying the assumptions of Theorem \ref{teolim}. 
Let $\lambda_1 + \lambda_2<n(1-\dfrac p{p_+})$ and let $\beta>-n$ be such that
\begin{equation*}
\max(0, \beta-n(\frac p{p_-}-1))<\lambda_1 + \lambda_2 (1+\beta/n)<n(1-\frac {p_-}{p_+})+\beta.
\end{equation*}
Then for $p_-<p<p_+$ it holds
 \begin{equation}\label{boundmorreypwlim}
\|g\|_{\mathcal L^{p,\lambda}(|x|^\beta)}\le C \|f\|_{\mathcal L^{p,\lambda}(|x|^\beta)}.
\end{equation}
If \eqref{hyp3} holds for $p_0=p_-$, then the conclusion \eqref{boundmorreypwlim} is satisfied for $p=p_-$.
\end{theorem}

\subsection{Embeddings}
A byproduct of the proof of Theorem \ref{teo1K} is the inequality
\begin{equation*}
\int_{\rn} f M(h^sw^s\chi_B)^{1/s}\le C \|f\|_{\mathcal L^{p,\lambda}(|x|^\alpha w)},
\end{equation*}
which implies the continuous embedding of the Morrey space $\mathcal L^{p,\lambda}(|x|^\alpha w)$ into the Lebesgue space $L^1(M(h^sw^s\chi_B)^{1/s})$. Therefore,
\begin{equation*}
\mathcal L^{p,\lambda}(|x|^\alpha w)\subset \bigcup_{u\in A_1} L^1(u).
\end{equation*}
For $p>1$, the argument at the end of the proof of Theorem \ref{teo1K} can be used to write 
\begin{equation*}
\mathcal L^{p,\lambda}(|x|^\alpha w)\subset \bigcup_{u\in A_1} L^q(u)\subset \bigcup_{u\in A_q} L^q(u),
\end{equation*}
for some $q\in (1,p)$. It was proved in \cite{KMM16} that the last term is independent of $q>1$ (see another proof in \cite{DR18}). 

By this embedding, if we have an operator $T$ that satifies the inequalities \eqref{hyp2} for some $p_0$ in the sense that the pairs considered in the statement are of the form $(|f|, |Tf|)$, then $T$ is already defined on the Morrey spaces for which the estimate has been proved in the theorem. 

A similar argument applies to the spaces appearing in Theorems \ref{teo2K}, \ref{teolim} and \ref{teo2Klim}.

\subsection{Applications}
Many known operators satisfy the assumptions of considering the pairs $(|f|,|Tf|)$, where $T$ is the operator. Therefore, all of them are bounded on weighted Morrey spaces of the form $\mathcal L^{p,\lambda}(|x|^\alpha w)$ for $p>1$ and $\alpha$ and $\lambda$ as stated. For the particular case of power weighted Morrey spaces, $\mathcal L^{p,\lambda}(|x|^\beta)$, the results of Theorem \ref{teo2K} apply. 

A list of operators fulfilling the requirements are:    
\begin{itemize}
\item Calder\'on-Zygmund operators and their associated maximal operators defined by the truncated integrals;
\item classical square functions (Lusin area integral, Littlewood-Paley $g$-function, etc.);
\item rough singular integrals with kernel p.v.\,$\Omega(x/|x|)|x|^{-n}$ with $\Omega\in L^\infty(\sn)$ and vanishing integral;
\item Bochner-Riesz operators at the critical index;
\item commutators of any linear operator of the previous list with a BMO function.
\end{itemize}
Except for the last case, all the other operators satisfy weighted weak-type estimates with $A_1$ weights for $p=1$. As a consequence of the theorems they satisfy weighted weak-type estimates also in the corresponding Morrey spaces. To apply the theorems in this case one uses the pairs $(|f|, t\chi_{\{|Tf|>t\}})$ with $p_0=1$. Also vector-valued  inequalities in weighted Morrey spaces are deduced directly from the abstract formulation of the theorems in terms of pairs of functions. 

See \cite{DR18} for more details and references about the $A_p$-weighted inequalities of the mentioned operators. In the same paper the reader will find operators to which the versions given in Theorems \ref{teolim} and \ref{teo2Klim} apply.

\section*{Acknowledgements} 
The first author is supported by the grants MTM2017-82160-C2-2-P of the Ministerio de Econom\'{\i}a y Competitividad (Spain) and FEDER, and IT1247-19 of the Basque Gouvernment. He would like to thank Yoshihiro Sawano for sharing with him the paper \cite{NST19} before it was published.

Part of this work was carried out during several visits of the second author to the University of the Basque Country (UPV/EHU). He would like to acknowledge the partial support of the grants IT-641-13 and IT1247-19 of the Basque Gouvernment.


\end{document}